\newcommand{\dataversione}
{September 29, 2017 (revised)}

\documentclass[11pt,reqno,a4paper,final]{amsart}
\usepackage[T1]{fontenc}
\usepackage{graphicx}%\graphicspath{{./figure/}}
\usepackage{fancyhdr}
\usepackage[notref,notcite]{showkeys}
\usepackage{mathrsfs}
\usepackage{amssymb}
\usepackage{hyperref}\hypersetup{colorlinks=true, citecolor=blue}

\numberwithin{equation}{section}

\newtheoremstyle{mytheorem}
{10pt}% measure of space to leave above the theorem. E.g.: 3pt
{10pt}% measure of space to leave below the theorem. E.g.: 3pt
{\it}% name of font to use in the body of the theorem
{}% measure of space to indent
{\bf}% name of head font
{.}% punctuation between head and body
{ }% space after theorem head; " " = normal interword space
{\thmnumber{#2.~}\thmname{#1}\thmnote{~\rm#3}}% Manually specify head

\newtheoremstyle{myremark}
{10pt}% measure of space to leave above the theorem. E.g.: 3pt
{10pt}% measure of space to leave below the theorem. E.g.: 3pt
{\rm}% name of font to use in the body of the theorem
{}% measure of space to indent
{\bf}% name of head font
{.}% punctuation between head and body
{ }% space after theorem head; " " = normal interword space
{\thmnumber{#2.~}\thmname{#1}\thmnote{~\rm#3}}% Manually specify head

\newtheoremstyle{myparagraph}
{10pt}% measure of space to leave above the theorem. E.g.: 3pt
{10pt}% measure of space to leave below the theorem. E.g.: 3pt
{\rm}% name of font to use in the body of the theorem
{}% measure of space to indent
{\bf}% name of head font
{.}% punctuation between head and body
{ }% space after theorem head; " " = normal interword space
{\thmnumber{#2.~}\thmname{#1}\thmnote{#3}}% Manually specify head

\theoremstyle{mytheorem}
\newtheorem{theorem}[subsection]{Theorem}
\newtheorem{lemma}[subsection]{Lemma}
\newtheorem{corollary}[subsection]{Corollary}
\newtheorem{proposition}[subsection]{Proposition}

\theoremstyle{myremark}
\newtheorem{remark}[subsection]{Remark}
\newtheorem*{remark*}{Remark}

\theoremstyle{myparagraph}
\newtheorem{parag}[subsection]{}
\newtheorem*{parag*}{}

\makeatletter
\def\@secnumfont{\sc}
\def\section{\@startsection%
{section}%	name: section/subsection/etc.
{1}%			level: 1 for section/2 for subsection/etc.
\z@{1.5\linespacing\@plus .2\linespacing}% 	vertical skip before
  {.7\linespacing}%	vertical skip after
  {\normalfont\sc\centering}}%	style
\makeatother
\makeatletter
\renewenvironment{proof}[1][\proofname]{\par 
  \pushQED{\qed}% 
  \normalfont \topsep10\p@\@plus6\p@\relax 
  \trivlist 
  \item[\hskip\labelsep 
    \bfseries 
    #1\@addpunct{.}]\ignorespaces 
}{% 
  \popQED\endtrivlist\@endpefalse 
} 
\providecommand{\proofname}{Proof}
\makeatother

\newcommand{\footnoteb}[1]{\footnote{~#1}}
\newenvironment{itemizeb}
{\begin{itemize}\itemsep=2pt\leftskip -5 pt}
{\end{itemize}}
\newcommand{\R}{\mathbb{R}}
\newcommand{\F}{\mathscr{F}}
\newcommand{\Haus}{\mathscr{H}}
\newcommand{\Leb}{\mathscr{L}}
\newcommand{\eps}{\varepsilon}
\newcommand{\bd}{\partial}

\newcommand{\Golab}{Go{\l}\k{a}b}
\newcommand{\Len}{\mathrm{length}}
\newcommand{\diam}{\mathrm{diam}}
\newcommand{\dist}{\mathrm{dist}}
\newcommand{\Lip}{\mathrm{Lip}}
\newcommand{\sgn}{\mathrm{sgn}}

\begin{document}

	% We do not use the maketitle command.
	% In the next lines we define the heading for the 
	% first page, then title, authors' names, and so on...
	% acknowledgements and affiliations are at the end of the paper

\thispagestyle{empty}
~\vskip -1.1 cm

	%
	% heading of first page
	%
{\footnotesize\noindent 
version:~\dataversione%
\,\raisebox{4pt}{$\dagger$}%
\let\thefootnote\relax\footnotetext{\raisebox{3pt}{$\dagger$}~%
Compared to the published version, we have added two pictures, 
corrected the proof of Lemma~\ref{s-lemmachiave1} and a few small 
mistakes, and modified the definition of $\delta$-partition in 
Subsection~\ref{s-deltapartitions}, in order to make the statement 
of Proposition~\ref{s-lunghezza2} stronger
(consequently, we have also modified Lemma~\ref{s-deltapartitions2}).}%
\hfill 
\emph{Nonlinear Analysis} 153 (2017), 35-55\par
\hfill DOI:~\href{http://dx.doi.org/10.1016/j.na.2016.10.012}%
{10.1016/j.na.2016.10.012} \par
}

\vspace{1.5 cm}

	%
	% title
	%
{\centering\Large\bf
On the structure of continua with finite length\par
\smallskip
and \Golab's semicontinuity theorem
\par
}

\vspace{.6 cm}

	%
	% authors' names
	%
\centerline{\sc Giovanni Alberti and Martino Ottolini}

\vspace{.6 cm}

	%
	% dedication
	%
{\centering\small\sl 
Dedicated to Nicola Fusco on the occasion of his 60th birthday\\
}

\vspace{.8 cm}

	%
	% abstract, keywords and MSC numbers
	%
{\rightskip 1 cm
\leftskip 1 cm
\parindent 0 pt
\footnotesize
{\sc Abstract.}
The main results in this note concern the characterization of the length 
of continua%
\,\raisebox{3pt}{\tiny1}\let\thefootnote\relax%
\footnote{\raisebox{3pt}{\scriptsize1}~
As usual, a \emph{continuum} is a connected compact metric 
space (or subset of a metric space), and the \emph{length} of a set 
is its one-dimensional Hausdorff measure $\Haus^1$.}  
(Theorems~\ref{s-lunghezza})
and the parametrization of continua with finite length (Theorem~\ref{s-canopara}).
Using these results we give two independent and relatively elementary proofs 
of \Golab's semicontinuity theorem.
\Golab's semicontinuity theorem.
\par
}

\section{Introduction}
Let $\F$ be the class of all continua $K$ contained in $\R^d$, 
endowed with the Hausdorff distance. 
A classical result due to S.~{\Golab} (see \cite{Go}, Section~3, 
or \cite{Fa}, Theorem~3.18)
states that the length, that is, the function  $K\mapsto\Haus^1(K)$, 
is lower semicontinuous on $\F$.
Variants of this semicontinuity result, together with well-known 
compactness properties of $\F$, play a key role in the proofs of 
several existence results in the Calculus of Variations, 
from optimal networks \cite{PS} to image segmentation \cite{DMMS} 
and quasi-static evolution of fractures \cite{DMT}.
In particular, \Golab's theorem has been extended to general metric spaces
in \cite{AT}, Theorem~4.4.17, and \cite{PS}, Theorem~3.3.%
\footnoteb{The proof in \cite{AT} is actually incomplete; 
the missing steps were given in \cite{PS}.}

\medskip
It should be noted that none of the proofs of \Golab's theorem
mentioned above is completely elementary. 
On the other hand, the counterpart of this result for paths, 
namely that the length of a path $\gamma:[0,1]\to X$ is lower 
semicontinuous with respect to the pointwise convergence of paths, 
is elementary and almost trivial.
This sharp contrast is due to the fact that the definitions of length 
of a path and of one-dimensional Hausdorff measure of a set are 
utterly different, even though they aim to describe (essentially) 
the same geometric quantity. 
More precisely, the length of a path, being defined as a supremum 
of finite sums which are clearly continuous, is naturally lower 
semicontinuous, while the definition of Hausdorff measure is based 
on Caratheodory's construction, and is designed to achieve 
$\sigma$-subadditivity, not semicontinuity.

\medskip
In this note we point out a couple of relations/similarities 
between the one-dimensional Hausdorff measure of continua and the 
length of paths, which we then use to give two independent (and 
relatively elementary) proofs of \Golab's theorem.
We think, however, that these results are interesting in their 
own right.

Firstly, in Theorem~\ref{s-lunghezza} we show that for 
every continuum $X$ there holds
\[
\Haus^1(X) = \sup \bigg\{ \sum_i \diam(E_i) \bigg\} 
\, , 
\]
where the supremum is taken over all finite families $\{E_i\}$
of disjoint \emph{connected} subsets of $X$. 
(Note the resemblance with the definition of length of a path.)

Secondly, in Theorem~\ref{s-canopara} we show that every continuum $X$ with 
finite length admits some sort of canonical parametrization; 
more precisely, there exists a path $\gamma:I\to X$ 
with length equal $2\,\Haus^1(X)$ which ``goes through  
almost every point of $X$ twice, once moving in a direction, 
and once moving in the opposite direction'',
the precise statement requires some technical definitions 
and is postponed to Section~\ref{s4}. 

\medskip
This paper is organized as follows: 
Sections~\ref{s2} and \ref{s4} 
contain the two results mentioned above
(Theorems~\ref{s-lunghezza} and~\ref{s-canopara}) 
and the corresponding proofs of \Golab's theorem.
Section~\ref{s3} contains a review of some basic facts about paths 
with finite length in a metric space which are used in Section~\ref{s4}, 
and can be skipped by the expert reader. 
This review is self-contained and limited in scope; 
a more detailed presentation of the theory of paths 
with finite length in metric spaces can be found in \cite{AT}, 
Chapter~4,  
while continua with finite length have been studied in detail in~\cite{EH} 
(see also \cite{Fre}).

Since the results described in this paper are rather 
elementary (in particular Theorem~\ref{s-lunghezza}), 
we strove to keep the exposition self-contained, 
and avoid in particular the use of advanced results 
from Geometric Measure Theory.
On the other hand, proofs are sometimes
just sketched, with all steps clearly indicated
but many details left to the reader.

\section{A characterization of length}
\label{s2}
The main results in this section are the characterizations of the length 
of sets with countably many connected components (and in particular 
of continua) given in Theorem~\ref{s-lunghezza} and Proposition~\ref{s-lunghezza2}.
Using the former result we give our first proof of \Golab's theorem
(Theorem~\ref{s-golab}).

\begin{parag}[Notation]
\label{s-metric}
Through this paper $X$ is a metric space endowed 
with the distance $d$.
Given $x\in X$ and $E, E'$ subsets of $X$ we set:
\begin{itemizeb}\leftskip .7 cm\labelsep=.2 cm
\item[$B(x,r)$]
\emph{closed} ball with center $x$ and radius $r>0$;
\item[$\diam(E)$]
diameter of $E$, i.e., 
$\sup\{ d(x,x') \, : \, x,x'\in E\}$;
\item[$\dist(x,E)$]
distance between $x$ and $E$, i.e.,
$\inf \{ d(x,x') \, : \, x'\in E\}$;
\item[$\dist(E,E')$]
distance between $E$ and $E'$, 
i.e., $\inf \{ d(x,x') \, : \, x\in E, \, x'\in E'\}$;%
\footnoteb{Since the infimum of the empty set is $+\infty$,
$\dist(E,E')=+\infty$ if either $E$ or $E'$ are empty.\label{footnote3}}

%
%\item[$B(E,r)$]
%$:=\inf \{ x\in X \, : \, \dist(x,E) \le r\}$, closed $r$-neighbourhood
%of $E$, with $r>0$;
%
\item[$d_H(E,E')$]  
Hausdorff distance between $E$ and $E'$, i.e.,
the infimum of all $r \ge 0$ s.t.\ $\dist(x,E')\le r$ for every $x\in E$ 
and $\dist(x',E)\le r$ for every $x'\in E'$;
\item[$\Lip(f)$]
Lipschitz constant of a map $f$ between metric spaces;
\item[$|E|$]
$=\Leb^1(E)$, Lebesgue measure of a Borel set $E$ contained in $\R$.
\end{itemizeb}
\end{parag}

\begin{parag}[Hausdorff measure]
\label{s-haus}
For every set $E$ contained in $X$, the one-dimensional 
Hausdorff measure of $E$ is defined by
\[
\Haus^1(E) 
:= \sup_{\delta>0} \Haus^1_\delta(E) 
=\lim_{\delta\to 0} \Haus^1_\delta(E) 
\, , 
\]
where, for every $\delta \in (0,+\infty]$, 
\[
\Haus^1_\delta(E) 
:= \inf \bigg\{ \sum_i \diam(E_i) \bigg\}
\, , 
\]
the infimum being taken over all countable families 
$\{E_i\}$ of subsets of $X$ which cover $E$ and satisfy
$\diam(E_i)\le\delta$.
\end{parag}

\begin{remark}
\label{s-remhaus}
Among the many properties of $\Haus^1$ we recall the following ones.

\smallskip
(i)~$\Haus^1$ is a $\sigma$-subadditive set function 
(that is, an outer measure on $X$) and is $\sigma$-additive
on Borel sets. Moreover $\Haus^1$ agrees with the 
(outer) Lebesgue measure $\Leb^1$ when $X=\R$.

\smallskip
(ii)~Given a Lipschitz map $f:X\to Y$, 
for every set $E$ contained in $X$ there holds 
$\Haus^1(f(E)) \le \Lip(f) \, \Haus^1(E)$.

\smallskip
(iii)~If $\Haus^1_\delta(E)=0$ for some $\delta\in(0,+\infty]$ then 
$\Haus^1(E)=0$.
\end{remark}

\begin{parag}[The set function $\boldsymbol{L_\delta}$]
\label{s-conmeas}
For every $\delta\in (0,+\infty]$ and every set $E$ in $X$ we define
\[
L_\delta(E) := \sup \bigg\{ \sum_i \diam(E_i) \bigg\} 
\, ,   
\]
where the supremum is taken over all finite, disjoint families 
$\{E_i\}$ of \emph{continua} contained in $E$ with $\diam(E_i)\le\delta$. 
\end{parag}

\begin{theorem}
\label{s-lunghezza}
Let $E$ be a subset of $X$ which is locally compact 
and has countably many connected components.%
\footnoteb{A connected component of $E$ 
is any element of the class of nonempty connected subsets
of $E$ which is maximal with respect to inclusion; 
the connected components are closed in $E$, disjoint, 
and cover $E$ (for more details see \cite{Eng}, Chapter~6).} 
Then, for every $\delta \in (0,+\infty]$,
\begin{equation}
\label{e-lunghezza}
\Haus^1(E) = L_\delta(E)
\, .
\end{equation}
\end{theorem}

\begin{remark}
(i)~The assumption that $E$ has countably many connected components cannot 
be dropped. Indeed $L_\delta(E)=0$ for every totally disconnected set $E$, 
and there are examples of such sets with $\Haus^1(E)>0$, 
even compact and contained in~$\R$.

\smallskip
(ii)~Theorem~\ref{s-lunghezza}, together with Lemma~\ref{s-diam},
implies the following weaker statement: 
for any set $E$ as above, $\Haus^1(E)$ agrees with the supremum of 
$\sum_i \diam(E_i)$ over all finite disjoint families $\{E_i\}$ of
connected subsets of $E$.  
Concerning this identity, it is not clear if the assumption that 
$E$ is locally compact can be weakened or even removed.
The role of compactness in our proof is briefly 
discussed in Remark~\ref{s-remcomp}.
\end{remark}

Using Theorem~\ref{s-lunghezza} we can actually show that $\Haus^1(E)$
can be approximated by $\sum_i \diam(E_i)$ using \emph{any} partition 
of $E$ made of connected subsets $E_i$ with sufficiently small diameters. 
For a precise statement we need the following definition.

\begin{parag}[$\boldsymbol{\delta}$-Partitions]
\label{s-deltapartitions}
Let $E$ be a subset of $X$ and let $\delta\in(0,+\infty]$.
We say that a countable family $\{E_i\}$ of subsets of $E$ 
is a \emph{$\delta$-partition} of $E$ if the sets $E_i$ 
are Borel, connected, $\Haus^1$-essentially disjoint 
(i.e., $\Haus^1(E_i\cap E_j)=0$ for every $i\ne j$), 
cover all of $E$ except a subset $E'$ with $\Haus^1(E')\le 0$, 
and satisfy $\diam(E_i)\le\delta$.

If $E$ is locally compact, has finite length 
and countably many connected components, 
then Theorem~\ref{s-lunghezza}
implies that there exist $\delta$-partitions for every $\delta>0$. 
\end{parag}

\begin{proposition}
\label{s-lunghezza2}
Let $E$ be a subset of $X$. Then every $\delta$-partition $\{E_i\}$ of $E$ satisfies
\begin{equation}
\label{e-lunghezza2a}
\Haus^1(E) 
\ge \sum_i \Haus^1(E_i)
\ge \sum_i \diam(E_i)
\, .
\end{equation}
If in addition $E$ is locally compact and has countably 
many connected components, then for every $m<\Haus^1(E)$ 
there exists $\delta_0>0$ such that every $\delta$-partition
$\{E_i\}$ of $E$ with $\delta\le\delta_0$ satisfies
\begin{equation}
\label{e-lunghezza2b}
\sum_i \diam(E_i) \ge m
\, .
\end{equation}
\end{proposition}

Using Theorem~\ref{s-lunghezza} we can also prove the following 
version of \Golab's theorem.

\begin{theorem}
\label{s-golab}
For every $m=1,2,\dots$, let $\F_m$ be the class of all nonempty
compact subsets of $X$ with at most $m$ connected components. 
Then the function $K\mapsto \Haus^1(K)$ is lower-semicontinuous 
on $\F_m$ endowed with the Hausdorff distance.
\end{theorem}

\begin{remark}
\label{s-remgolab}
The statement of \Golab's theorem is often restricted to the case $m=1$, 
and in the metric setting reads as follows
(cf.~\cite{AT}, Theorem~4.4.17): let be given a sequence 
of continua $K_n$, contained in a \emph{complete} 
metric space $X$, which converge in the Hausdorff distance to 
some closed set $K$; then $K$ is a continuum, and 
$\liminf\Haus^1(K_n) \ge \Haus^1(K)$.
The assumption that $X$ is complete is needed here to ensure 
that the limit $K$ is compact and connected, 
but not to prove the semicontinuity of length.
\end{remark}

\bigskip
The rest of this section is devoted to the proofs
of Theorems~\ref{s-lunghezza} and \ref{s-golab}, 
and Proposition~\ref{s-lunghezza2}.
We begin with the proof of Theorem~\ref{s-lunghezza}; 
the key estimate is contained in Lemma~\ref{s-basiclemma}.

\begin{lemma}
\label{s-diam}
Let $E$ be a connected set in $X$.
Then $\Haus^1(E) \ge \diam(E)$. 
\end{lemma}

\begin{proof}
It suffices to prove that $\Haus^1(E) \ge d(x_0,x_1)$ for every 
$x_0,x_1\in E$. Let indeed $f:X\to\R$ be the function defined by
$f(x):=d(x,x_0)$. Then 
\[
\Haus^1(E) \ge |f(E)| =\diam(f(E)) \ge |f(x_1)-f(x_0)| = d(x_1,x_0)
\, ,
\]
where the first inequality follows from Remark~\ref{s-remhaus}(ii) 
(and $\Lip(f)=1$), while the first equality follows 
from the fact that $f(E)$ is an interval (because $E$ is connected).
\end{proof}

\begin{lemma}
\label{s-easyineq}
For every set $E$ in $X$ and $\delta>0$ there holds 
$\Haus^1(E) \ge L_\delta(E)$. 
\end{lemma}

\begin{proof}
Consider any family 
$\{E_i\}$ as in the definition of $L_\delta(E)$:
Lemma~\ref{s-diam} yields
\[
\Haus^1(E) \ge \sum_i \Haus^1(E_i) \ge \sum_i\diam(E_i)
\, , 
\]
and we obtain $\Haus^1(E) \ge L_\delta(E)$ 
by taking the supremum over all~$\{E_i\}$.
\end{proof}

\begin{lemma}
\label{s-additivity}
Let $E$ be a subset of $X$ and let
$\{E_i\}$ be the family of all connected components of~$E$. Then
$L_\delta(E) = \sum_i L_\delta(E_i)$
for every $\delta>0$.%
\footnoteb{The sum at the right-hand side of this equality 
is defined as the supremum of all finite subsums, and is well 
defined even if the family $\{E_i\}$ is uncountable.}
\end{lemma}

The proof of this lemma is straightforward, and we omit it.

\begin{lemma}
\label{s-concompbordo}
Let $U$ be a nonempty compact set in~$X$, and let $F$ be a 
connected component of $U$.
If $F\cap\bd U=\varnothing$ then $F$ is also a connected 
component of $X$. Accordingly, if $X$ is connected and $F\ne X$ 
then $F\cap\bd U \ne \varnothing$.
\end{lemma}

\begin{proof}
Let $\F$ be the family of all sets $A$ such that $F \subset A \subset U$ 
and $A$ is open and closed in $U$. Then $\F$ is closed
by finite intersection, and $F$ agrees with the intersection of all $A\in\F$
(see \cite{Eng}, Theorem~6.1.23).

If $F\cap\bd U=\varnothing$ then the 
intersection of the compact sets $A\cap \bd U$ with $A\in\F$ is empty, 
which implies that $A\cap \bd U$ is empty for at least one $A\in\F$.%
\footnoteb{The basic fact behind this assertion is that every family 
of compact sets with empty intersection admits a finite
subfamily with empty intersection.}
This means that $F$ is the intersection of all $A\in\F$ such that
$A\cap \bd U = \varnothing$. 
Note that these sets $A$ are open and closed in~$X$, and then 
$F$ is connected and agrees with the intersection of a family 
of open and closed sets.
This implies that $F$ is a connected component of $X$. 
\end{proof}

\begin{corollary}
\label{s-concompball}
Let $E$ be a connected set in $X$, let $B=B(x,r)$ be 
a ball with center $x\in E$ such that $E\cap B$ is compact 
and $E\setminus B\ne\varnothing$, 
and let $F$ be the connected component of $E\cap B$ that contains
$x$. Then $\Haus^1(B\cap E)\ge\Haus^1(F)\ge r$.
\end{corollary}

\begin{proof}
By applying Lemma~\ref{s-concompbordo} with $E$ and $E\cap B$ in 
place of $X$ and $U$, we obtain that $F$ intersects $\bd B$.
Then $\diam(F) \ge r$, and Lemma~\ref{s-diam} yields $\Haus^1(F) \ge r$.
\end{proof}

\begin{remark}
\label{s-remcomp}
The compactness assumptions 
in Lemma~\ref{s-concompbordo} and Corollary~\ref{s-concompball}
are both necessary.
Indeed it is possible to construct a bounded connected set $E$ in $\R^2$
and a ball $B$ with center $x\in E$ such that $E\setminus B\ne\varnothing$, 
but the connected component $F$ of $E\cap B$ that contains $x$ consists just 
of the point $x$; in particular $F \cap \bd B=\varnothing$, 
and $\Haus^1(F)=0$.
\end{remark}

\begin{lemma}
\label{s-basiclemma}
Let $E$ be a set in $X$ which is connected and locally compact. 
Then $\Haus^1_\delta(E) \le L_\delta(E)$ for every $\delta>0$.
\end{lemma}

\begin{proof}
We can clearly assume that $L_\delta(E)$ is finite.
We fix for the time being $\eps>0$, and
choose a finite disjoint family $\{E_i\}$ of continua 
contained in $E$ with $\diam(E_i)\le\delta$ such that
\begin{equation}
\label{e-2.3}
\sum_i \diam(E_i) \ge L_\delta(E)-\eps
\, .
\end{equation}

Next we set $E':= \smash{ E\setminus\big( \cup_i E_i \big) }$.
Since the union of all $E_i$ is closed and $E$ is locally compact, for every 
$x\in E'$ we can find a ball $B(x,r)$ with radius $r\le \delta/10$ 
such that $E\cap B(x,r)$ is compact and contained in $E'$. 
Using Vitali's covering lemma (cf.~\cite{AT}, Theorem~2.2.3), 
we can extract from this family of balls
a subfamily of disjoint balls $B_j=B(x_j,r_j)$ 
such that the balls $B_j':=B(x_j,5r_j)$ cover~$E'$. 

Then the balls $B_j'$ together with the sets $E_i$
cover the set $E$, and since their diameters do not exceed $\delta$, 
the definition of $\smash{ \Haus^1_\delta(E) }$ yields
\begin{equation}
\label{e-2.4}
\Haus^1_\delta(E) 
\le \sum_i \diam(E_i) + \sum_j \diam(B_j')
\le L_\delta(E) + 10\, \sum_j r_j
\, .
\end{equation}

On the other hand, by Corollary~\ref{s-concompball}, for every $j$ 
we can find a closed, connected set $F_j$ contained in $B_j\cap E$ with 
diameter at least $r_j$. Since the balls $B_j$ are disjoint
and contained in $E'$, 
we have that the sets $F_j$ together with the
sets $E_i$ form a disjoint family of continua contained in $E$
with diameters at most $\delta$, and therefore, using the definition 
of $L_\delta(E)$ and \eqref{e-2.3},
\[
L_\delta(E) 
\ge \sum_i \diam(E_i) + \sum_j \diam(F_j) 
\ge L_\delta(E) -\eps + \sum_j r_j
\, , 
\]
which implies $\eps \ge \sum_j r_j$. 
Hence \eqref{e-2.4} yields
$\Haus^1_\delta(E) \le L_\delta(E) + 10\,\eps$,
and the proof is complete
because $\eps$ is arbitrary.
\end{proof}

\begin{proof}[Proof of Theorem~\ref{s-lunghezza}]
By Lemma~\ref{s-easyineq}, it suffices to prove that
\begin{equation}
\label{e-lunghezza1}
L_\delta(E) \ge \Haus^1(E)
\, .
\end{equation}
We assume first that $E$ is connected.
In this case Lemma~\ref{s-basiclemma} 
and the definition of $L_\delta$ in Subsection~\ref{s-conmeas} 
yield 
\[
L_\delta(E) 
\ge L_{\delta'}(E)
\ge \Haus^1_{\delta'}(E) 
\quad
\text{for every $0 < \delta' \le \delta$,}
\]
and we obtain \eqref{e-lunghezza1} by taking the limit as $\delta'\to 0$.

If $E$ is not connected, then \eqref{e-lunghezza1} holds for every connected 
component of $E$, and we obtain that it holds for $E$ as well using
Lemma~\ref{s-additivity}, the subadditivity of $\Haus^1$, and the fact that
$E$ has countably many connected components.
\end{proof}

The next lemma is used in the proof of Proposition~\ref{s-lunghezza2}.

\begin{lemma}
\label{s-deltapartitions2}
Let $F$ be a continuum in $X$,
let $\{E_i\}$ be a countable family of connected 
subsets of $X$, and let $E'$ be the union of all $E_i$.
Then
\[
\Haus^1(F\setminus E') + \sum_i \diam(E_i) \ge \diam(F)
\, .
\]
\end{lemma}

\begin{proof}
Take $x_0,x_1\in F$ such that $d(x_0,x_1)=\diam(F)$, 
and let $f:X\to\R$ be the Lipschitz function given by $f(x):=d(x,x_0)$. 
Then
\begin{align*}
  \Haus^1(F\setminus E') + \sum_i \diam(E_i) 
& \ge |f(F\setminus E')| + \sum_i \diam(f(E_i)) \\
&   = |f(F\setminus E')| + \sum_i |f(E_i)| 
  \ge |f(F)|
  \ge \diam(F)
  \, ,
\end{align*}
where the first inequality follows from the fact that $\Lip(f)=1$, 
for the equality we use that each $f(E_i)$ is an
interval, for the second inequality we use that 
the sets $f(E_i)$ together with $f(F\setminus E')$ cover
$f(F)$, and the last inequality follows from the fact that 
$f(F)$ is an interval that contains $f(x_0)=0$
and $f(x_1)=d(x_0,x_1)=\diam(F)$.
\end{proof}

\begin{proof}[Proof of Proposition~\ref{s-lunghezza2}]
To prove \eqref{e-lunghezza2a} we use
the definition of $\delta$-partition and estimate 
$\Haus^1(E_i)\ge \diam(E_i)$ (see Lemma~\ref{s-diam}).

To prove the second part of the statement,
we first choose $\delta_1>0$ such that $\Haus^1(E) > m+\delta_1$, 
and use Theorem~\ref{s-lunghezza} to find finitely many 
disjoint continua $F_j$ contained in $E$ such that 
\begin{equation}
\label{s-lunghezza2-1}
\sum_j \diam(F_j) \ge m+\delta_1
\, .
\end{equation}
Then we take $\delta_0$ with $0<\delta_0\le\delta_1$
such that $\dist(F_j,F_k)>\delta_0$ for every $j\ne k$.

Consider now any $\delta$-partition $\{E_i\}$
of $E$ with $\delta\le\delta_0$. 
Let $E'$ be the union of all $i$.
For every $j$, let $I_j$ the the collection 
of all indices~$i$ such that $E_i$ intersects $F_j$, 
and let $E'_j$ be the union of all $E_i$ with $i\in I_j$. 
By the choice of $\delta_0$
the collections $I_j$ are pairwise disjoint, and therefore
\begin{align*}
  \Haus^1(E\setminus E') 
& + \sum_i \diam(E_i) \ge {} \\ 
& \ge \sum_j \bigg[ \Haus^1(F_j\setminus E'_j) 
                    + \sum_{i\in I_j} \diam(E_i) \bigg]
  \ge \sum_j \diam(F_j)
\, ,
\end{align*}
where the last inequality follows from 
Lemma~\ref{s-deltapartitions2}.
Since $\{E_i\}$ is a $\delta$-partition of $E$ 
we obtain that
\begin{equation}
\label{s-lunghezza2-2}
\delta + \sum_i \diam(E_i) \ge \sum_j \diam(F_j)
\, ,
\end{equation}
and putting together \eqref{s-lunghezza2-1}, \eqref{s-lunghezza2-2}
and the fact that $\delta\le\delta_0\le\delta_1$ 
we obtain \eqref{e-lunghezza2b}.
\end{proof}

We now pass to the proof of Theorem~\ref{s-golab}.

\begin{parag}[$\boldsymbol\delta$-chains 
and $\boldsymbol\delta$-connected sets]
\label{s-deltachain}
Given $\delta>0$, a \emph{$\delta$-chain} in $X$ is any finite 
sequence of points $\{x_i: \, i=0,\dots,n\}$ contained in $X$
such that $d(x_{i-1}, x_i)\le\delta$ for every $i>0$. 
We call $x_0$ and $x_n$ \emph{endpoints} of the $\delta$-chain, 
and we say that the $\delta$-chain connects $x_0$ and~$x_n$.
The \emph{length} of the $\delta$-chain is 
\[
\Len(\{x_i\}) 
:= \sum_{i=1}^n d(x_{i-1}, x_i)
\, .
\]
Finally, we say that a set $E$ in $X$ is $\delta$-connected 
%(or more precisely, connected by $\delta$-chains) 
if every couple of points $x,x'\in E$ is connected by 
a $\delta$-chain contained in~$E$.
\end{parag}

\begin{lemma}
\label{s-deltaconn}
If $E$ is a connected set in $X$ then it is $\delta$-connected
for every $\delta>0$. 
\end{lemma}

\begin{proof}
Fix $x\in E$ and let $A_x$
be the set of all points $x'\in E$ which are connected
to $x$ by a $\delta$-chain contained in $E$. We must show that $A_x=E$.

One easily checks that:
\begin{itemizeb}
\item
$A_x$ is closed in $E$ and contains $x$; 
\item
if $x' \in A_x$ then $B(x',\delta)\cap E$ is contained in $A_x$;
thus $A_x$ is open in $E$.
\end{itemizeb}
Since $A_x$ is nonempty, open and closed in $E$, and $E$ is connected, 
we conclude that $A_x=E$, as desired.
\end{proof}

\begin{lemma}
\label{s-lemmachiave0}
Let $K$ be a compact set in $X$ with at most $m$ connected components,
which contains a $\delta$-connected subset $K'$.
Then 
\begin{equation}
\label{e-stimachiave0}
\Haus^1(K) \ge \diam(K') - m\delta
\, .
\end{equation}
\end{lemma}

\begin{proof}
We can assume that $m$ is finite.
We then take $x_0,x_1\in K'$ such that
\begin{equation}
\label{e-stima0.1}
\ell:=d(x_0,x_1) \ge \diam(K') - \delta
\, ,
\end{equation}
and let $H:=f(K)$ where $f:X\to\R$ is defined by 
$f(x):=d(x,x_0)$. It is easy to check that:
\begin{itemizeb}
\item[(a)]
$\Lip(f)=1$, and therefore
$\Haus^1(K) \ge |H| \ge \ell - \big| (0,\ell)\setminus H \big|$;
\item[(b)]
the sets $f(K')$ and $H$ contain $0=f(x_0)$ and $\ell=f(x_1)$;
\item[(c)]
$H$ has at most $m$ connected components because
so does $K$;
\item[(d)]
the set $f(K')$ is $\delta$-connected because $K'$ is 
$\delta$-connected and $\Lip(f)=1$.
\end{itemizeb}
Statements~(b) and (c) imply that the open set $(0,\ell)\setminus H$ 
has at most $m-1$ connected components, while statements 
(b) and (d) imply that each of these connected components 
has length at most $\delta$; in particular 
\begin{equation}
\label{e-stima0.3}
\big| (0,\ell)\setminus H \big| \le (m-1)\delta
\, .
\end{equation}
Using the estimate in (a), \eqref{e-stima0.1},
and \eqref{e-stima0.3} we finally obtain \eqref{e-stimachiave0}.
\end{proof}

\begin{lemma}
\label{s-lemmachiave1}
Let $K$ be a compact set in $X$ with at most $m$ connected components, 
let $K'$ be a $\delta$-connected subset of $K$, 
and let $U$ be a closed set in $X$ which contains $K'$ 
and satisfies
\begin{equation}
\label{e-ipotesichiave}
\dist(K',\bd U) \ge r
\end{equation}
for some $r>0$.%
\,\footnoteb{If $\bd U$ is empty then $\dist(K',\bd U)=+\infty$ (cf.~Footnote~\ref{footnote3}) 
and then \eqref{e-ipotesichiave} holds for every $r>0$.}
Then 
\begin{equation}
\label{e-stimachiave}
\Haus^1(K\cap U) \ge \Big( 1 - \frac{\delta}{r} \Big) \, \diam(K') - m\delta
\, .
\end{equation}
\end{lemma}

\begin{proof}
%[Proof of Lemma~\ref{s-lemmachiave1}]
We can clearly assume that both $\Haus^1(K\cap U)$ and $m$ are finite.

Let $\{K_i\}$ be the collection of all connected components 
of $K\cap U$ that intersect $K'$, and let $N$ be their number.
We claim that $N$ is finite, and more precisely
\begin{equation}
N \le m + \frac{1}{r} \Haus^1(K\cap U)
\, . 
\label{e-stima1}
\end{equation}
To prove this estimate, 
note that the components $K_i$ which do not intersect $\bd U$ are 
also connected components of $K$ (apply Lemma~\ref{s-concompbordo}
with $K$, $U\cap K$ and $K_i$ in place of $X$ and $U$ and $F$, respectively), 
and therefore their number is at most $m$. 

There are now two cases:
either there are no components $K_i$ that intersect $\bd U$,%
\footnoteb{This includes the case when $\bd U$ is empty.}
and then $N \le m$, which implies~\eqref{e-stima1},
or there are components $K_i$ that intersect $\bd U$.
Since these components intersect also $K'$ they satisfy
\[
\Haus^1(K_i) \ge \diam(K_i) \ge r 
\]
(use Lemma~\ref{s-diam} and assumption~\eqref{e-ipotesichiave}),
and since their number is at least $N-m$ we obtain 
\[
\Haus^1(K\cap U) 
\ge (N - m)r 
\, ,
\]
which implies~\eqref{e-stima1}.

\begin{figure}[h]
\begin{center}
  \includegraphics[scale=1.2]{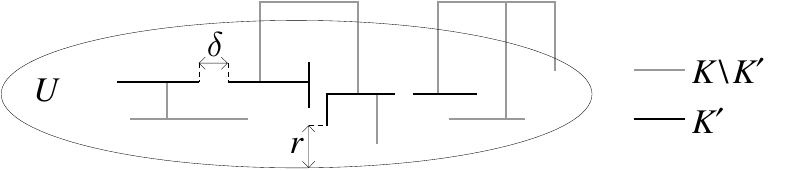}  
  \caption{In this example, $K$ has three connected components, 
  while $K\cap U$ has six connected components, four of which intersect $K'$.} 
  \label{figure1}
\end{center}
\end{figure}

Let now $K''$ be the union of all components $K_i$. 
Then $K''$ contains $K'$ and is compact (because it is 
a finite union of closed subsets of $K$), 
and applying Lemma~\ref{s-lemmachiave0} with $K''$ 
in place of $K$ we obtain
\[
\Haus^1(K\cap U) 
\ge \Haus^1(K'') 
\ge \diam(K') -N\delta
\, .
\]
Using estimate \eqref{e-stima1} we then get
\[
\Big(1 +\frac{\delta}{r} \Big)
\Haus^1(K\cap U) 
\ge \diam(K') -m\delta
\, ,
\]
which in turn implies \eqref{e-stimachiave}.
\end{proof}

\begin{proof}[Proof of Theorem~\ref{s-golab}]
We must show that for every sequence of compact sets $K_n\in \F_m$ 
that converge in the Hausdorff distance to some $K\in\F_m$ there holds 
$\liminf\Haus^1(K_n) \ge \Haus^1(K)$.
Taking into account Theorem~\ref{s-lunghezza} it suffices to prove that
\begin{equation}
\label{e-golab1}
\liminf_{n\to +\infty} \Haus^1(K_n) 
\ge \sum_i \diam(E_i)
\end{equation}
for every finite family $\{E_i\}$ of 
disjoint continua contained in $K$.

Since the sets $E_i$ are compact and disjoint, 
we can find $r>0$ and a family of disjoint closed sets $U_i$ 
such that each $U_i$ contains $E_i$ and satisfies%
\,\footnoteb{If there is only one $E_i$ we can take $U_i:=X$, 
and then $\dist(E_i,X\setminus U_i)=+\infty$ (Footnote~\ref{footnote3}).}
\begin{equation}
\label{e-2.12}
\dist(E_i,X\setminus U_i) \ge r
\, .
\end{equation}
Then \eqref{e-golab1} follows by showing that, for every $i$,
\begin{equation}
\label{e-golab2}
\liminf_{n\to +\infty} \Haus^1(K_n\cap U_i) 
\ge \diam(E_i)
\, .
\end{equation}

Let us fix $i$ and choose $\delta$ such that $0<\delta<r$. 
Since $E_i$ is connected, it is also $\delta$-connected
(Lemma~\ref{s-deltaconn}) and therefore it contains a 
$\delta$-chain $\{x_0,\dots,x_k\}$ with
\[
d(x_0,x_k) \ge \diam(E_i) - \delta
\, .
\]

Consider now any $n$ such that $d_H(K_n,K) \le \delta$
(that is, any $n$ sufficiently large).
By the definition of Hausdorff distance, for every 
point $x_j$ in the $\delta$-chain we can choose a point 
$y_j\in K_n$ with $d(x_j,y_j)\le\delta$, and we set 
$K'_n:=\{y_0,\dots,y_k\}$.
One readily checks that
\begin{itemizeb}
\item[(a)]
$K'_n$ is a $3\delta$-chain, 
and therefore is $3\delta$-connected;
\item[(b)]
$\diam(K'_n) \ge d(y_0,y_k) \ge d(x_0,x_k) -2\delta \ge \diam(E_i)-3\delta$;
\item[(c)]
$\dist(K'_n,\bd U_i) 
\ge \dist(K'_n,X\setminus U_i) 
\ge \dist(E_i,X\setminus U_i) - \delta 
\ge r-\delta$.%
\footnoteb{Note that this chain of inequalities holds also
when $\bd U_i$ and $X\setminus U_i$ are empty.}
\end{itemizeb} 
We can then apply Lemma~\ref{s-lemmachiave1}
with $K_n$, $K'_n$, $U_i$, $3\delta$ and $r-\delta$ 
in place of $K$, $K'$, $U$, $\delta$ and $r$, 
respectively, and obtain
\begin{align*}
\Haus^1(K_n \cap U_i) 
& \ge \Big( 1 - \frac{3\delta}{r-\delta} \Big) \, \diam(K'_n) - 3m\delta \\
& \ge \Big( 1 - \frac{3\delta}{r-\delta} \Big) \, (\diam(E_i)-3\delta) - 3m\delta
  \, .
\end{align*}
To obtain \eqref{e-golab2} we take the liminf 
as $n\to+\infty$ and then the limit as $\delta\to 0$.
\end{proof}

\section{Basic properties of paths in metric spaces}
\label{s3}
In this section we recall some basic facts
concerning paths with finite length, focusing
in particular on two results that will be used in the following 
section, namely Propositions~\ref{s-pathconn}
and \ref{s-lengthformula}.
Both statements are well-known at least in the Euclidean case. 

\begin{parag}[Paths]
\label{s-path}
A \emph{path} in $X$ is a continuous map $\gamma:I\to X$ where $I=[a_0,a_1]$ is 
a closed interval. Then $x_0:=\gamma(a_0)$ and $x_1:=\gamma(a_1)$ are 
called \emph{endpoints} of $\gamma$, and we say that $\gamma$ connects 
$x_0$ to $x_1$.
If $x_0=x_1$ we say that $\gamma$ is \emph{closed}.

The \emph{multiplicity} of $\gamma$ at a point $x\in X$ is the number (possibly equal
to $+\infty$) 
\[
m(\gamma,x):=\#(\gamma^{-1}(x)) 
\, .
\]

The \emph{length} of $\gamma$ is
\[
\Len(\gamma) 
= \Len(\gamma,I) 
:= \sup \bigg\{ \sum_{i=1}^n d \big( \gamma(t_{i-1}),\gamma(t_i) \big) \bigg\}
\]
where the supremum is taken over all $n=1,2,\dots$ and all increasing sequences 
$\{t_0,\dots,t_n\}$ contained in $I$.%
\footnoteb{The length of $\gamma$ is sometimes called \emph{variation} and 
denoted by $\mathrm{Var}(\gamma,I)$; paths with finite length are called 
\emph{rectifiable}.}

The length of $\gamma$ relative to a closed interval $J$ contained in $I$, 
denoted by $\Len(\gamma,J)$, is the length of the restriction of $\gamma$ 
to $J$.
If $\gamma$ has finite length it is sometimes useful to consider 
the \emph{length measure} associated to $\gamma$, namely the (unique) 
positive measure $\mu_\gamma$ on $I$ which satisfies
\[
\mu_\gamma([t_0,t_1]) = \Len(\gamma,[t_0,t_1])
\quad\text{for every $[t_0,t_1]\subset I$.}
\]

We say that $\gamma$ is a \emph{geodesic} if it has finite length
and minimizes the length among all paths 
with the same endpoints.

We say that $\gamma$ has \emph{constant speed} if there exists 
a finite constant $c$ such that
\[
\Len(\gamma,[t_0,t_1]) = c(t_1-t_0)
\quad\text{for every $[t_0,t_1] \subset I$.}
\]

An (orientation preserving) \emph{reparametrization} of 
$\gamma$ is any path $\gamma':I'\to X$ of the form
\[
\gamma' = \gamma\circ\tau
\]
where $\tau:I'\to I$ is an increasing homeomorphism.
\end{parag}

\begin{remark}
\label{s-pathrem}
Here are some elementary (and mostly well-known) facts.

\smallskip
(i)~The length is lower semicontinuous with respect to 
the pointwise convergence of paths. More precisely, 
given a sequence of paths $\gamma_n:I\to X$ which converge
pointwise to $\gamma:I\to X$, it is easy to check that
\[
\Len(\gamma) \le \liminf_{n\to+\infty} \Len(\gamma_n)
\, .
\]

\smallskip
(ii)~Every path $\gamma:I\to X$ with finite length $\ell$, 
which is not constant on any subinterval of $I$,
admits a Lipschitz reparametrization 
$\gamma':[0,1]\to X$ with constant speed~$\ell$, 
namely $\gamma':=\gamma\circ\sigma^{-1}$
where $\sigma:I\to [0,1]$ is the homeomorphism given by 
\[
\sigma(t):=\frac{1}{\ell}\Len(\gamma,[a_0,t])
\quad\text{for every $t\in I=[a_0,a_1]$.}
\]

\smallskip
(iii)~If $\gamma$ is constant on some subinterval 
of $I$ then the function $\sigma$ defined above is continuous, surjective, 
but not injective. However, we can still consider the left-inverse 
$\tau$ defined by
\[
\tau(s):=\min\{t: \, \sigma(t)=s \}
\quad\text{for every $s\in [0,1]$,}
\]
and even though $\tau$ is not continuous, one can check that
$\gamma':=\gamma\circ\tau$ is a continuous path with 
constant speed~$\ell$, and $m(\gamma',x)=m(\gamma,x)$
for all points $x\in X$ except countably many.

\smallskip
(iv)~If $\gamma$ is Lipschitz then 
$\Len(\gamma,J) \le \Lip(\gamma) \, |J|$
for every interval $J$ contained in $I$,
and more generally 
$\mu_\gamma(E) \le \Lip(\gamma) \, |E|$
for every Borel set $E$ contained in $I$.
Thus the length measure $\mu_\gamma$ is absolutely 
continuous with respect to the Lebesgue measure on $I$, 
and more precisely it can be written as $\smash{ \mu_\gamma=\rho\,\Leb^1 }$
with a density $\rho:I\to\R$ such that 
$0 \le \rho \le \Lip(\gamma)$~a.e.

\smallskip
(v)~If $\gamma$ has constant speed $c$ then $\Lip(\gamma)=c$,
$\Len(\gamma) = c\,|I|$ and $\mu_\gamma=c\,\Leb^1$. 
Conversely, it is easy to check that 
if $\Lip(\gamma)\,|I| \le \Len(\gamma) < +\infty$ then 
$\gamma$ has constant speed $c=\Lip(\gamma)=\Len(\gamma)/|I|$.
\end{remark}

\begin{remark}
\label{s-pathrem2}
The following result is worth mentioning,
even though it will not be used in the following:
if $\gamma$ is Lipschitz and $\rho$
is taken as in Remark~\ref{s-pathrem}(iv), then
for a.e.~$t\in I$ there holds
\begin{equation}
\label{e-metricderivative}
\lim_{h\to 0} \frac{d\big( \gamma(t+h),\gamma(t) \big)}{|h|}
=\lim_{h\to 0^+} \frac{\Len\big(\gamma,[t-h,t+h] \big)}{2h}
=\rho(t)
\, .
\end{equation}
The second equality in \eqref{e-metricderivative}
is a straightforward consequence of Lebesgue's differentiation theorem, 
while the first one is not immediate and will not be proved here.
The first limit in \eqref{e-metricderivative}
is called \emph{metric derivative} of $\gamma$ 
(see~\cite{AT}, Definition~4.1.2 and Theorem~4.1.6). 
\end{remark}

%\begin{parag}[Reparametrizations]
%\label{s-repar}
%Let be given a path $\gamma:I\to X$.
%An (orientation preserving) \emph{reparametrization} of 
%$\gamma$ is any subpath $\gamma'=\gamma\circ\tau$ 
%which satisfies the following additional conditions:
%
%\begin{itemizeb}
%\item
%$\gamma$ is constant on $[\tau(s),\tau(s^+)]$ 
%for every discontinuity point $s$ of $\tau$;
%\item
%$\gamma$ is constant on $[\tau(s_1),t_1]$, 
%where $s_1$ and $t_1$ are the right endpoints of the intervals
%$I'$ and $I$, respectively.
%\end{itemizeb}
%
%In particular, if $\gamma$ is constant on no interval,
%then $\tau:I'\to I$ must be an increasing homeomorphism.%
%
%\footnoteb{The reason for considering $\tau$ which is not an 
%homeomorphism is Proposition~\ref{s-constspeed}.}
%
%One easily checks that if $\gamma'$ is a reparametrization 
%of $\gamma$ then $\Len(\gamma') = \Len(\gamma)$.
%\end{parag}
%
%\begin{proposition}
%\label{s-constspeed}
%Every path $\gamma$ with finite length admits a reparametrization 
%$\gamma'$ with constant speed. 
%
%Moreover the domain of 
%$\gamma'$ can be any prescribed interval $[s_0,s_1]$, 
%and therefore $\Lip(\gamma')=\Len(\gamma)/(s_1-s_0)$.
%\end{proposition}

We can now state the main results of this section.

\begin{proposition}
\label{s-pathconn}
Let $X$ be a continuum with $\Haus^1(X)<+\infty$, 
and let $x \ne x'$ be points in $X$.
Then $x$ and $x'$ are connected by an injective
geodesic $\gamma:[0,1]\to X$ with constant speed 
and length $\ell\le \Haus^1(X)$.
\end{proposition}

If $X$ is a subset of $\R^n$, this statement can be found 
for example in \cite{Fa}, Lemma~3.12. 
A slightly more general version of this statement 
(in the metric setting) can be found in 
\cite{AT}, Theorem~4.4.7. 
For the sake of completeness we give a proof below, 
which follows essentially the one in~\cite{Fa}.

\begin{proposition}
\label{s-lengthformula}
Let $\gamma:I\to X$ be a path with finite length. 
Then the multiplicity $m(\gamma,\cdot):X\to[0,+\infty]$ 
is a Borel function and 
\begin{equation}
\label{e-lengthformula}
\Len(\gamma) = \int_X m(\gamma,x) \, d\Haus^1(x)
\, .
\end{equation}
In particular $m(\gamma,x)$ is finite for $\Haus^1$-a.e.~$x\in X$.
\end{proposition}

\begin{remark}
(i)~Formula~\eqref{e-lengthformula} can be viewed as the one-dimensional 
area formula in the metric setting, in particular if coupled with the 
existence of the metric derivative, see Remark~\ref{s-pathrem2}.

\smallskip
(ii)~Formula~\eqref{e-lengthformula} can easily re-written in  
local form: for every Borel function $f:X\to [0,+\infty]$ there holds
\begin{equation}
\label{e-lengthformula2}
\int_I f(\gamma(t)) \, d\mu_\gamma(t) 
= \int_X f(x) \, m(\gamma,x) \, d\Haus^1(x)
\, .
\end{equation}
This means that the push-forward of the length measure 
$\mu_\gamma$ according to the map $\gamma$ agrees with the 
measure $\Haus^1$ on $X$ multiplied by the density $m(\gamma,\cdot)$;
in short $\gamma_\#(\mu_\gamma)=m(\gamma,\cdot) \, \Haus^1$.
\end{remark}

\bigskip
The rest of this section is devoted to the proofs of  
Propositions~\ref{s-pathconn} and \ref{s-lengthformula}.
We begin with some preliminary lemmas.

\begin{lemma}
\label{s-chainconn}
Take $X, x,x'$ as in Proposition~\ref{s-pathconn}.
Then, for every $\delta>0$, $x$ and $x'$ are connected by a 
$\delta$-chain $\{x_i: \, i=0,\dots, n\}$  (see Subsection~\ref{s-deltachain}) 
such that 
\begin{equation}
\label{e-chainconn}
\Len(\{x_i\}) \le 4\,\Haus^1(X)
\, .
\end{equation}
\end{lemma}

\begin{proof}
We can assume $\delta<d(x,x')$, otherwise
it suffices to take the $\delta$-chain consisting
just of the points $x,x'$ and use Lemma~\ref{s-diam}
to obtain~\eqref{e-chainconn}.

By Lemma~\ref{s-deltaconn},
$x$ and $x'$ are connected a $\delta$-chain
$\{x_i\}$, and 
possibly removing some points from the chain, 
we can further assume that
\begin{equation}
\label{e-pathconn1.1}
d(x_i,x_j)>\delta
\quad\text{if $|j-i|\ge 2$.}
\end{equation}
Consider now the balls $B_i:=B(x_i,\delta/2)$ with $i=0,\dots,n$, 
and note that
$\Haus^1(B_i)\ge\delta/2$ by Corollary~\ref{s-concompball}, 
while \eqref{e-pathconn1.1} implies that $B_i$ and $B_j$
do not intersect if $|j-i|\ge 2$, which means that
every point in $X$ belongs to at most two balls 
in the family $\{B_i\}$. 
Using these facts and the estimate $\delta \ge d(x_{i-1},x_i)$
we obtain
\[
2\, \Haus^1(X) 
\ge \sum_{i=0}^n \Haus^1(B_i)
\ge \sum_{i=0}^n \frac{\delta}{2}
\ge \frac{1}{2} \, \Len(\{x_i\})
\, .
\qedhere
\]
\end{proof}

\begin{lemma}
\label{s-lengthformula2}
For every path $\gamma:I\to X$ there holds
$\Haus^1(\gamma(I)) \le \Len(\gamma)$.
\end{lemma}

\begin{proof}
It suffices to show that $\Haus^1_\delta(\gamma(I)) \le \Len(\gamma)$
for every $\delta>0$ (cf.~Subsection~\ref{s-haus}).
Using the continuity of $\gamma$, we partition $I$ into finitely many
closed intervals $I_i$ with disjoint interiors so that
\[
\diam(\gamma(I_i)) \le \delta
\quad\text{for every $i$.}
\]
Using the definition of $\Haus^1_\delta$ and the fact that
$\diam(\gamma(I_i)) \le \Len(\gamma, I_i)$, we obtain
\[
\Haus^1_\delta(\gamma(I)) 
\le \sum_i \diam(\gamma(I_i)) 
\le \sum_i \Len(\gamma, I_i)
=\Len(\gamma)
\, .
\qedhere
\]
\end{proof}

\begin{lemma}
\label{s-lengthformula3}
If $\gamma:I\to X$ is an injective path then
$\Len(\gamma) = \Haus^1(\gamma(I))$.
\end{lemma}

\begin{proof}
By Lemma~\ref{s-lengthformula2} and the definition of
length it suffices to show that for every increasing sequence
$\{t_0,\dots,t_n\}$ in $I$ there holds
\[
\sum_{i=1}^n d \big(\gamma(t_{i-1}),\gamma(t_i) \big)
\le \Haus^1(\gamma(I))
\, .
\]
Since the sets $E_i:=\gamma([t_{i-1},t_i])$ are
connected, then
$\diam(E_i) \le \Haus^1(E_i)$ (Lemma~\ref{s-diam}),
and since $\gamma$ is injective, 
the intersection $E_i\cap E_j$ contains at most one point
for every $i\ne j$, and in particular is $\Haus^1$-null. 
Hence
\[
\sum_{i=1}^n d \big(\gamma(t_{i-1}),\gamma(t_i) \big)
\le \sum_{i=1}^n \diam(E_i) 
\le \sum_{i=1}^n \Haus^1(E_i) 
\le \Haus^1(\gamma(I))
\, .
\qedhere
\]
\end{proof}

\begin{proof}[Proof of Proposition~\ref{s-pathconn}]
The idea is simple: for every $\delta>0$ 
we take the (almost) shortest $\delta$-chain 
$\smash{ \{x^\delta_i\} }$ that connects $x$ and $x'$, 
and consider the set $\Gamma_\delta$ of all couples
$\smash{ (t^\delta_i,x^\delta_i) } \in [0,1]\times X$ 
with suitably chosen times $\smash{ t^\delta_i }\in [0,1]$. 
Passing to a subsequence, we can assume that the 
compact sets $\Gamma_\delta$ converge in the Hausdorff distance
to some limit set $\Gamma$ as $\delta\to 0$; 
we then show that $\Gamma$ is the graph of a
path $\gamma:[0,1]\to X$ with the desired properties.

\smallskip
We set $I:=[0,1]$ and $L:=\Haus^1(X)$.
The proof is divided in several steps.

\medskip
\emph{Step~1: construction of the $\delta$-chain $\{x^\delta_i\}$.}
Fix $\delta>0$, and let $\F_\delta$ be the class of 
all $\delta$-chains with initial point $x$ and final 
point $x'$, and let $L_\delta$ be the infimum of the length 
over all $\delta$-chains in~$\F_\delta$.
By Lemma~\ref{s-chainconn} we know that $\F_\delta$ 
is not empty and $L_\delta \le 4L$.

We then choose a $\delta$-chain $\smash{ \{x^\delta_i: \, i=0,\dots,n_\delta\} }$ 
in $\F_\delta$ whose length $\ell_\delta$ satisfies
\begin{equation}
\label{e-pathconn2.0}
L_\delta \le \ell_\delta \le L_\delta+\delta \le 4L+\delta
\, .
\end{equation}

\medskip
\emph{Step~2: construction of the set $\Gamma_\delta$.}
Fix $\delta>0$ and let $\smash{ \{x^\delta_i: \, i=0,\dots,n_\delta\} }$
be the $\delta$-chain chosen in Step~1.
We can clearly assume that the points $\smash{x^\delta_i}$ 
are all different, and find an increasing sequence 
of numbers $\smash{ t^\delta_i }$ with $\smash{ i=0,\dots,n_\delta }$ 
such that the first one is $0$, the last one is $1$, and the differences
$\smash{ t^\delta_i-t^\delta_{i-1} }$ are proportional to the distances 
$\smash{ d(x^\delta_{i-1},x^\delta_i) }$. 
This means that
\begin{equation}
\label{e-pathconn2.1}
\frac{ d(x^\delta_{i-1},x^\delta_i) }{ t^\delta_i-t^\delta_{i-1} } 
= \ell_\delta
\end{equation}
for every $i=1,\dots,n_\delta$,
and in particular we have that 
\begin{equation}
\label{e-pathconn2.2}
t^\delta_i-t^\delta_{i-1} 
= \frac{ d(x^\delta_{i-1},x^\delta_i) }{ \ell_\delta }
\le \frac{\delta}{\ell_\delta} 
\le \frac{\delta}{d(x,x')}
\, .
\end{equation}
Finally, we set 
\[
\Gamma_\delta:= \big\{ (t^\delta_i,x^\delta_i) : \, i=0,\dots,n_\delta \big\}
\, .
\]

\medskip
\emph{Step~3: construction of the set $\Gamma$.}
The sets $\Gamma_\delta$ defined in Step~2 are contained in the compact 
metric space $[0,1]\times X$, and by Blaschke's selection 
theorem (see for example \cite{AT}, Theorem~4.4.15) we have that, 
possibly passing to a subsequence, 
they converge in the Hausdorff distance as $\delta\to 0$ 
to some compact set $\Gamma$ contained 
in $[0,1]\times X$.

\medskip
\emph{Step~4: $\Gamma$ is the graph of a Lipschitz path $\gamma:I\to X$.}
Formula \eqref{e-pathconn2.1} implies 
that each $\Gamma_\delta$ is the graph of a map $\gamma_\delta$ 
from a subset of $I$ to $X$ with Lipschitz constant~$\ell_\delta$.  
This immediately implies that $\Gamma$ is the graph 
of a Lipschitz map $\gamma$ from a subset of $I$ to $X$ with
$\Lip(\gamma) \le \ell$ where (recall~\eqref{e-pathconn2.0})
\begin{equation}
\label{e-pathconn2.3}
\ell
:=\liminf_{\delta\to 0} \ell_\delta 
= \liminf_{\delta\to 0} L_\delta
\le 4L
\, .
\end{equation}
Moreover the projection of $\Gamma_\delta$ on $I$ is the set
$I_\delta:=\{ \smash{t^\delta_i} \}$, and taking 
into account estimate \eqref{e-pathconn2.2} and the fact 
that $I_\delta$ contains $0$ and $1$, 
we get that $I_\delta$ converges to $I$ in the Hausdorff distance
as $\delta\to0$. This implies that the projection of $\Gamma$ on $I$ 
is $I$ itself, which means that the domain of $\gamma$ is $I$.

\medskip
\emph{Step~5: $\gamma$ connects $x$ and $x'$.}
Since $\smash{ x^\delta_0=x }$ and $\smash{ x^\delta_{n_\delta}=x' }$
for every $\delta>0$, each $\Gamma_\delta$ contains the points 
$(0,x)$ and $(1,x')$, and therefore so does $\Gamma$, which means 
that $\gamma(0)=x$ and $\gamma(1)=x'$.

\medskip
\emph{Step~6: $\ell \le \Len(\gamma)$.}
For every $\delta>0$ we can extract from the image of $\gamma$ a 
$\delta$-chain that connects $x$ and $x'$
and has length at most $\Len(\gamma)$. This implies that
$\ell_\delta \le \Len(\gamma)$ (cf.~Step~1), and using 
\eqref{e-pathconn2.3} we obtain the claim. 

\medskip
\emph{Step~7: $\gamma$ has constant speed, and $\Len(\gamma)=\ell$.}
By Step~4 and Step~6 we have that $\Lip(\gamma) \le \ell \le \Len(\gamma)$. 
Then the claim follows from Remark~\ref{s-pathrem}(v).

\medskip
\emph{Step~8: $\gamma$ is a geodesic.}
Let $\gamma'$ be any path connecting $x$ and $x'$.
Arguing as in Step~6 we obtain $\ell \le \Len(\gamma')$, 
which implies $\Len(\gamma) \le \Len(\gamma')$ by Step~7.

\medskip
\emph{Step~9: $\gamma$ is a injective.}
Assume by contradiction that there exists $t_0\in I$ and $s>0$ such that 
$\gamma(t_0)=\gamma(t_0+s)$. 
Then the path $\gamma':[0,1-s]\to X$ 
defined by
\[
\gamma'(t):=
\begin{cases}
  \gamma(t) & \text{if $0 \le t\le t_0$,} \\
  \gamma(t+s) & \text{if $t_0 < t\le 1-s$.} \\
\end{cases}
\]
is well-defined, connects $x$ and $x'$, and has length $\ell(1-s)$,
which is strictly smaller than the length of $\gamma$, 
contrary to the fact that $\gamma$ is a geodesic. 

\medskip
\emph{Step~10: $\Len(\gamma) \le \Haus^1(X)$.} 
Apply Lemma~\ref{s-lengthformula3}.
\end{proof}

We pass now to the proof of Proposition~\ref{s-lengthformula}.

\begin{parag}[Piecewise regular paths]
\label{s-piecewisereg}
Let $I$ be a closed interval.
We say that a finite family 
$\{I_i\}$ is a \emph{partition} of $I$ 
if the $I_i$ are closed intervals contained in $I$, 
have pairwise disjoint interiors, and cover $I$,
and we say that a path $\gamma:I\to X$ is \emph{piecewise regular}
on the partition $\{I_i\}$ 
if it is either constant or injective on each $I_i$.
\end{parag}

\begin{lemma}
\label{s-pathapprox}
Let $\gamma:I\to X$ be a path with finite length, 
and let $\{I_i\}$ be a partition of $I$.
Then there exists a path $\gamma' : I \to X$ 
such that:
\begin{itemizeb}
\item[(i)]
$\gamma'$ is piecewise regular on the partition $\{I_i\}$;
\item[(ii)]
$\gamma'$ agrees with $\gamma$ at the endpoints of each $I_i$
and $\gamma'(I_i)\subset\gamma(I_i)$;
\item[(iii)]
$\Len(\gamma',I_i) 
= \Haus^1(\gamma'(I_i)) 
\le \Haus^1(\gamma(I_i)) 
\le \Len(\gamma,I_i)$ for every $i$.
\end{itemizeb}
\end{lemma}

\begin{proof}
We define $\gamma'$ on each interval $I_i=[a_i,a_i']$ as follows: 
\begin{itemizeb}
\item
if $\gamma(a_i)=\gamma(a'_i)$ we let
$\gamma'$ be the constant path $\gamma(a_i)$;
\item
if $\gamma(a_i)\ne\gamma(a'_i)$, we let $\gamma'$ be
any injective path from $I_i$ to $X':=\gamma(I_i)$ 
which connects $\gamma(a_i)$ to $\gamma(a'_i)$
(such path exists because $X'$ has finite length, 
cf.~Proposition~\ref{s-pathconn}).
\end{itemizeb}
The path $\gamma'$ satisfies statements~(i) and (ii)
by construction, while (iii) follows from
Lemmas~\ref{s-lengthformula2} and~\ref{s-lengthformula3}.
\end{proof}

\begin{proof}[Proof of Proposition~\ref{s-lengthformula}]
The proof is divided in three cases.

\medskip
\emph{Case~1: $\gamma$ is injective.}
In this case the multiplicity $m(\gamma,\cdot)$ is 
the characteristic function of 
the compact set $\gamma(I)$, and therefore is Borel, 
while identity \eqref{e-lengthformula} follows from 
Lemma~\ref{s-lengthformula3}.

\medskip
\emph{Case~2: $\gamma$ is piecewise regular.}
We easily reduce to the previous case. 

\medskip
\emph{The general case.}
We choose a sequence of 
piecewise regular paths $\gamma_n:I\to X$ 
that approximate $\gamma$ in the following sense:
\begin{itemizeb}
\item[(a)]
$\gamma_n$ converge to $\gamma$ uniformly;
\item[(b)]
$\Len(\gamma_n) \le \Len(\gamma)$ for every $n$; 
\item[(c)] 
$m(\gamma_n,x) \le m(\gamma,x)$ for every $x\in X$
and every $n$;
\item[(d)] 
$m(\gamma_n,x) \to m(\gamma,x)$ as $n\to+\infty$
for every $x\in X$.
\end{itemizeb}
More precisely, we construct $\gamma_n$ as follows:
for every $n$ we choose a partition $\{I^n_i\}$ of $I$
so that 
\begin{equation}
\label{e-finezza}
\max_i \diam(I^n_i) \to 0
\quad\text{as $n\to+\infty$.}
\end{equation}
and then take $\gamma_n$ according to Lemma~\ref{s-pathapprox}.
Then statements~(a), (b), (c) and (d) can be readily derived 
from \eqref{e-finezza} and statements~(ii) and (iii) in 
Lemma~\ref{s-pathapprox}.

\smallskip
We can now prove that the multiplicity $m(\gamma,\cdot)$ 
is Borel and \eqref{e-lengthformula} holds.
The first part of this claim follows by the fact that
$m(\gamma,\cdot)$ agrees with the pointwise limit of the 
multiplicities $m(\gamma_n,\cdot)$ (statement~(d)), which are 
Borel measurable because the paths $\gamma_n$ fall into Case~2.
To prove \eqref{e-lengthformula}, note that 
statements~(a) and (b) and the semicontinuity of length 
imply that
\begin{itemizeb}
\item[(e)]
$\Len(\gamma_n) \to \Len(\gamma)$ as $n\to+\infty$, 
\end{itemizeb}
while statements~(c) and (d) and Fatou's lemma yield
\begin{itemizeb}
\item[(f)]
$\int_X m(\gamma_n,\cdot) \, d\Haus^1 \to \int_X m(\gamma,\cdot) \, d\Haus^1$
as $n\to+\infty$.
\end{itemizeb}
We already know that \eqref{e-lengthformula}
holds for each $\gamma_n$, 
and using statements~(e) and (f)
we can pass to the limit (as $n\to+\infty$) 
and obtain that \eqref{e-lengthformula} holds
for $\gamma$ as well.
\end{proof}

\section{Parametrizations of continua with finite length}
\label{s4}
In this section we address the following question: can we parametrize
a continuum $X$ by a single path $\gamma:I\to X$, and if yes, 
what can we require about $\gamma$?

First of all, note that in general a continuum cannot be 
parametrized by a one-to-one path, and not even by a path
with multiplicity equal to $1$ at almost every point
(take for example any network with a triple junction).%
\footnoteb{By network we mean here a \emph{connected} union of finitely many 
arcs (that is, images of injective Lipschitz paths) which intersect 
at most at the endpoints; a point which agrees with $n$ endpoints, $n\ge 3$, 
is called an $n$-junction.}
On the other hand, it is easy to see that every network can be 
parametrized by a closed path that goes through every
arc in the network twice, once in a direction and once 
in the opposite direction.

In Theorem~\ref{s-canopara} we show that 
something similar holds for every continuum $X$ with finite length, 
and more precisely there exists a closed path that \emph{goes through almost 
every point of $X$ twice, once in a direction and once in the opposite direction}.

Before stating the result, we must give a precise formulation of 
the requirement in italic. 
If $X$ is a network made of regular arcs of class $C^1$ in $\R^n$, 
we simply ask that $\gamma$ has multiplicity equal to $2$ and 
\emph{degree} equal to $0$ 
at every point of $X$ except junctions. The problem in extending this
condition to general continua is that the usual definition 
of degree cannot be easily adapted to the metric setting.
To get around this issue, in Subsection~\ref{s-degzero} we introduce 
a suitable weaker notion of path with degree zero.

\medskip
Unless further specification is made, 
in the following $X$ is a metric space.

\begin{parag}[Paths with degree zero]
\label{s-degzero}
Given a Lipschitz path $\gamma:I\to X$, a locally bounded 
Borel function $f:X\to\R$, and a Lipschitz function $g:X\to\R$, 
we introduce the notation
\begin{equation}
\label{e-intform}
\int_\gamma f \, dg
:= \int_I (f\circ\gamma) \, \frac{d}{dt} (g\circ\gamma) \, dt
\, .
\end{equation}
Note that $g\circ\gamma$ is Lipschitz, and therefore 
the derivative in the integral at the right-hand side is well-defined
at almost every~$t\in I$ and bounded in $t$, and  
the integral itself is well-defined.

We say that $\gamma$ has \emph{degree zero} (at almost every 
point of its image) if 
\begin{equation}
\label{e-degzero}
\int_\gamma f \, dg = 0
\quad\text{for every $f, g:X\to\R$ Lipschitz.}
\end{equation}
\end{parag}

\begin{remark}
\label{s-degzerorem}
(i)~A simple approximation argument shows that 
if $\gamma$ has degree zero then $\int_\gamma f \, dg = 0$
for every Lipschitz function $g:X\to\R$ and every bounded 
Borel function $f:X\to\R$.

\smallskip
(ii)~If $X$ is a finite union of oriented regular arcs in $\R^n$, 
or more generally an oriented $1$-rectifiable set, 
and $\gamma:I\to X$ is a Lipschitz path, then 
for $\Haus^1$-almost every $x\in X$ one can define the 
degree of $\gamma$ at $x$, denoted by
$\deg(\gamma,x)$.
Moreover for every $f,g:X\to\R$ there holds
\begin{equation}
\label{e-orientedareaformula}
\int_\gamma f \, dg
:= \int_X f(x) \, \frac{\bd g}{\bd\tau}(x) 
          \, \deg(\gamma,x) \, d\Haus^1(x)
\, , 
\end{equation}
where $\bd g/\bd\tau$ is the tangential derivative 
of $g$.
Using this formula it is easy to check that \eqref{e-degzero}
holds if and only if $\deg(\gamma,x)=0$ for $\Haus^1$-a.e.~$x\in X$.
This justifies the use of the expression ``path with degree zero''
in Subsection~\ref{s-degzero}.

\smallskip
(iii)~Formula \eqref{e-degzero} can be reinterpreted in the framework
of metric currents by saying that the push-forward according to 
$\gamma$ of the canonical $1$-current associated to the (oriented) 
interval $I$ is trivial.
\end{remark}
 
\begin{proposition}
\label{s-degzeroprop}
Let $\gamma:I\to X$ be a Lipschitz path, and let $f,g:X\to\R$ 
be Lipschitz functions. Then the following statements hold.
\begin{itemizeb}
\item[(i)]
\emph{[Invariance under reparametrization]}
Let $\tau:I'\to I$ be an increasing homeomorphism
such that $\gamma\circ\tau$ is Lipschitz.
Then 
\begin{equation}
\label{e-invariance}
\int_\gamma f \, dg = \int_{\gamma\circ\tau} f \, dg 
\, .
\end{equation}
In particular, $\gamma$ has degree zero 
if and only if $\gamma\circ\tau$ has degree zero.
\item[(ii)]
\emph{[Stability]}
Given a sequence of paths $\gamma_n:I\to X$ 
which are uniformly Lipschitz and converge uniformly to $\gamma$, 
then 
\[
\int_{\gamma_n} f \, dg  \to \int_\gamma f \, dg 
\quad\text{as $n\to+\infty$.}
\]
In particular, if each $\gamma_n$ has degree zero, 
then $\gamma$ has degree zero.
\item[(iii)]
\emph{[Parity]}
If $\gamma$ has degree zero then the multiplicity 
$m(\gamma,x)$ is finite and even 
for $\Haus^1$-a.e.~$x\in X$.
\end{itemizeb}
\end{proposition}

We can now state the main result of this section.

\begin{theorem}
\label{s-canopara}
Let $X$ be a continuum with finite length. 
Then there exists a path $\gamma: [0,1]\to X$ with the following properties:
\begin{itemizeb}
\item[(i)]
$\gamma$ is closed, Lipschitz, surjective, and has degree zero;
\item[(ii)]
$m(\gamma,x)=2$ for $\Haus^1$-a.e.~$x\in X$, and $\Len(\gamma)=2\,\Haus^1(X)$;
\item[(iii)]
$\gamma$ has constant speed, equal to $2\,\Haus^1(X)$.
\end{itemizeb}
\end{theorem}

\begin{remark}
\label{s-canopararem}
(i)~The existence of a Lipschitz surjective path $\gamma:[0,1]\to X$ with 
$\Lip(\gamma) \le 2\,\Haus^1(X)$ was first proved in \cite{Wa}.
Here we simply point out 
that $\gamma$ can be taken of degree zero.

\smallskip
(ii)~An immediate corollary of this result is that 
every continuum $X$ with finite length is a rectifiable
set of dimension $1$. 

\smallskip
(iii)~If $X$ is contained in $\R^n$, then one can apply 
Rademacher's differentiability theorem to the parametrization
$\gamma$ and prove with little effort that $X$ admits a tangent
line in the classical sense at $\Haus^1$-a.e.~point.
%
%\smallskip
%(iv)~The set of all points $x\in X$ where $m(\gamma,x)\ne 2$ 
%cannot be made too small. More precisely,
%there exists a continuum $X$ with the following property: for 
%every path $\gamma:I\to X$ that satisfies statement~(ii) above, 
%the set of all $x\in X$ such that $m(\gamma,x)=1$ 
%has Hausdorff dimension equal to one.
\end{remark}

\bigskip
The rest of this section is devoted to the proofs of
Proposition~\ref{s-degzeroprop} and Theorem~\ref{s-canopara}. 
At the end of the section we give
another proof of \Golab's theorem based on the latter.

\begin{parag}[Additional notation]
\label{s-orientednotation}
Let be given a Lipschitz path $\gamma:I\to X$
and a Lipschitz function $g:X\to\R$.
We write $h:=g\circ\gamma$, 
and denote by $N$ the set of all $s\in\R$ such that one of the following 
properties fails:
\begin{itemizeb}
\item[(a)]
the set $h^{-1}(s)$ is finite;
\item[(b)]
the derivative $\dot{h}$ exists and is not $0$
at every $t\in h^{-1}(s)$.
\end{itemizeb}
Thus for every $s\in\R\setminus N$ and every $x\in g^{-1}(s)$, 
the following sum is well-defined and finite:
\begin{equation}
\label{e-grado}
p(x) := \sum_{t\in\gamma^{-1}(x)} \sgn\big( \dot{h}(t) \big)
\, , 
\end{equation}
where, as usual, $\sgn(x):=1$ if $x>0$, $\sgn(x):=-1$ if $x<0$, and $\sgn(0)=0$.
%\[
%\sgn(x):=
%\begin{cases}
%	+1 & \text{if $x>0$,} \\
%	-1 & \text{if $x<0$.}
%\end{cases}
%\]
\end{parag}

\begin{lemma}
\label{s-orientedarea}
Take $\gamma$, $g$, $h$, $N$ and $p$ as in Subsection~\ref{s-orientednotation}. 
Then $|N|=0$ and for every Lipschitz function $f:X\to\R$ there holds
\begin{equation}
\label{e-orientedarea}
  \int_\gamma f \, dg
= \int\limits_{\R\setminus N} 
    \bigg[ \sum_{x\in g^{-1}(s)} f(x) \, p(x) \bigg] \, ds
\, .
\end{equation}
\end{lemma}

\begin{proof}
To prove that $|N|=0$ we write 
\[
N=N_0\cup h(E_0) \cup h(E_1)
\, , 
\]
where $N_0$ is the set of all $s\in\R$ such that $h^{-1}(s)$ is infinite, 
$E_0$ is the set of all $t\in I$ where the derivative of $h$ exists 
and is $0$, $E_1$ is the set of all $t\in I$ where the derivative of 
$h$ does not exists.

We observe now that $|N_0|=0$ and $|h(E_0)|=0$ 
by the one-dimensional area formula applied 
to the Lipschitz function $h:I\to\R$,%
\footnoteb
{The one-dimensional area formula we use reads as follows: if $h:I\to\R$ is 
Lipschitz and $f:I\to\R$ is either positive or in $L^1(I)$ then 
\[
\int_I f \, |\dot h| \, dt = \int\limits_\R \bigg[ \sum_{t\in h^{-1}(s)} f(t) \bigg] ds
\, .
\]
In particular $\int_I |\dot h| \, dt = \int_\R m(h,s) \, ds$ where 
$m(h,s)$ is the multiplicity of $h$ at $s$, which implies that $m(h,s)$ is finite 
for a.e.~$s\in\R$.}
while $|E_1|=0$ by Rademacher's theorem and then 
$|h(E_1)|=0$ because $h$ is Lipschitz.
We conclude that $|N|=0$.

\smallskip
Let us prove \eqref{e-orientedarea}. 
Using \eqref{e-intform}, the area formula, and that $|N|=0$,
we get
\begin{align*}
   \int_\gamma f \, dg
  =\int_I (f\circ\gamma) \, \dot{h} \, dt 
& =\int_I (f\circ\gamma) \, \sgn(\dot{h}) 
        \, |\dot{h}| \, dt \\
& =\int\limits_{\R\setminus N} \bigg[ 
             \sum_{t\in h^{-1}(s)} f(\gamma(t)) \, \sgn(\dot{h}(t)) 
         \bigg] \, ds
  \, , 
\end{align*}
and we obtain \eqref{e-orientedarea} by suitably rewriting the
sum within square brackets.
\end{proof}

\begin{proof}[Proof of Proposition~\ref{s-degzeroprop}(i)]
Given $g$ and $\gamma$, we take $h$, $N$ and $p$ as in 
Subsection~\ref{s-orientednotation}, and let $h'$ $N'$ and $p'$ be 
the analogous quantities where $\gamma$ is replaced 
by $\gamma\circ\tau$. 
Thanks to Lemma~\ref{s-orientedarea}, 
identity \eqref{e-invariance} can be proved by showing that
$p(x)=p'(x)$ for every $x$ such that $g(x)\notin\setminus(N\cup N')$.

Taking into account \eqref{e-grado} and the fact that $h'=h\circ\tau$, 
the identity $p(x)=p'(x)$ reduces to the following elementary statement:
given $t\in I$ such that the derivative of $h$ at $t$ exists and is nonzero, 
and the derivative of $h'=h\circ\tau$ at $t':=\tau^{-1}(t)$ exists and is nonzero, 
then these derivatives have the same sign (recall that $\tau$ is increasing). 
\end{proof}

\begin{proof}[Proof of Proposition~\ref{s-degzeroprop}(ii)]
In view of \eqref{e-intform} it suffices to show that 
\[
\int_I (f \circ \gamma_n) \, \frac{d}{dt} (g\circ\gamma_n) \, dt
\to
\int_I (f \circ \gamma) \, \frac{d}{dt} (g\circ\gamma) \, dt
\quad\text{as $n\to+\infty$.}
\]
This is an immediate consequence of the fact that the functions
$f \circ \gamma_n$ converge to $f \circ \gamma$ uniformly, 
and therefore strongly in $L^1(I)$,
while the derivatives of the functions $g\circ\gamma_n$ converge 
to the derivative of $g\circ\gamma$
in the weak* topology of $L^\infty(I)$. 
\end{proof}

The next lemmas are used in the proof of Proposition~\ref{s-degzeroprop}(iii).

\begin{lemma}
\label{s-quasimetricderiv}
Let $\gamma:I\to X$ be a path with finite length, and let 
$\mu_\gamma$ be the corresponding length measure (Subsection~\ref{s-path}).
Then for $\mu_\gamma$-a.e.~$t\in I$ there holds
\begin{equation}
\label{e-quasimetricderiv}
\rho(t):=\liminf_{r\to 0} \frac{ \diam\big( \gamma(B(t,r)) \big) }{ r } >0
\, .
\end{equation}
\end{lemma}

This lemma would be an immediate consequence of formula~\eqref{e-metricderivative}, 
which however we did not prove. The proof below is self-contained.

\begin{proof}
Let $E := \{t\in I: \, \rho(t)=0\}$.
We must prove that $\mu_\gamma(E)=0$.

Let $\eps>0$ be fixed for the time being.
For every $t\in E$ we can find $r_t>0$ such that 
the ball (i.e., centered interval) $B(t,r_t)$ 
is contained in $I$ and
\[
\diam\big( \gamma(B(t,r_t)) \big) \le \eps r_t
\, .
\]
Consider now the family $\F$ of all balls $B(t,r_t/5)$ with $t\in E$. 
Using Vitali's covering lemma (see for example \cite{AT}, 
Theorem~2.2.3), we can extract from $\F$
a subfamily of disjoint balls $B_j=B(t_j,r_j/5)$ 
such that the balls $B_j':=B(t_j,r_j)$ cover $E$. 
Thus the sets $\gamma(B'_j)$ cover $\gamma(E)$ and
can be used to estimate $\Haus^1_\infty(E)$ (see Subsection~\ref{s-haus}):
\begin{align*}
      \Haus^1_\infty(\gamma(E)) 
& \le \sum_j \diam(\gamma(B'_j)) \\
& \le \eps \sum_j r_j 
  = \frac{5\eps}{2} \sum_j |B_j|
  \le \frac{5\eps}{2} \,|I|
\end{align*}
(for the last inequality we used that the balls $B_j$ are disjoint
and contained in $I$).
Since $\eps$ is arbitrary, we obtain that $\Haus^1_\infty(\gamma(E))=0$ and
therefore $\Haus^1(\gamma(E))=0$ (cf.~Remark~\ref{s-remhaus}(iii)).
Using formula \eqref{e-lengthformula2} we finally get 
\[
\mu_\gamma(E) = \int_{\gamma(E)} m(\gamma,x)\, d\Haus^1(x) = 0
\, .
\qedhere
\]
\end{proof}

\begin{lemma}
\label{s-goodtest}
Let $\gamma:I\to X$ be a path with finite length, 
and let $E$ be a Borel subset of $\gamma(I)$ with $\Haus^1(E)>0$.
Then there exists a Lipschitz function $g:X\to\R$ such that
$|g(E)|>0$.
\end{lemma}

\begin{proof}
We can assume that $\gamma$ has constant speed~$1$
(Remark~\ref{s-pathrem}(iii)), which implies that
$\gamma$ has Lipschitz constant $1$ and the length 
measure $\mu_\gamma$ agrees with the Lebesgue measure
on $I$.

We set $F:=\gamma^{-1}(E)$ and $F':=I\setminus F$.
Since $E=\gamma(F)$ is not $\Haus^1$-negligible,
$F$ must have positive Lebesgue measure, and 
using Lemma~\ref{s-quasimetricderiv} and Lebesgue's density theorem 
we can find a point $t\in F$ where \eqref{e-quasimetricderiv} holds
and $F$ has density~$1$, and accordingly $F'$ has density $0$.

We define $g:X\to\R$ by $g(x) := d(x,\gamma(t))$, 
and set $h:=g\circ\gamma$.
By \eqref{e-quasimetricderiv} there exists 
$\delta>0$ such that $\diam(\gamma(B(t,r))) \ge 2\delta r$
for every ball $B(t,r)$ contained in~$I$. 
This implies that
\[
\sup_{t'\in B(t,r)} h(t') =
\sup_{t'\in B(t,r)} d(\gamma(t'),\gamma(t))
\ge \delta r
\, .
\]
Thus the interval $h(B(t,r))$ contains $[0,\delta r]$ and then
\begin{equation}
\label{e-goodtest1}
\big| h(B(t,r)) \big| \ge \delta r
\, .
\end{equation}
On the other hand, the fact that $h$ is Lipschitz 
and $F'$ has density $0$ at $t$ implies
\begin{equation}
\label{e-goodtest2}
\big| h (F'\cap B(t,r)) \big| = o(r)
\, .
\end{equation}
Finally, the inclusion 
\[
g(E) =h(F) 
\supset h(B(t,r)) \setminus h(F'\cap B(t,r))
\, ,
\]
together with estimates \eqref{e-goodtest1} and \eqref{e-goodtest2}, yields
\[
|g(E)| \ge \delta r - o(r)
\,  , 
\]
and we conclude by observing that $\delta r - o(r)>0$ 
for $r$ small enough.
\end{proof}

\begin{proof}[Proof of Proposition~\ref{s-degzeroprop}(iii)]
We already know that $m(\gamma,x)$ is finite for $\Haus^1$-a.e.~$x\in X$
(Proposition~\ref{s-lengthformula}).
Let then $E$ be the set of all $x\in X$ such that $m(\gamma,x)$ 
is finite and \emph{odd}, and assume by contradiction that $\Haus^1(E)>0$.

By Lemma~\ref{s-goodtest} there exists a Lipschitz function 
$g:X\to\R$ such that $|g(E)|>0$.  
Then we take $N$ and $p$ as in Subsection~\ref{s-orientednotation}, 
and let $f:X\to\R$ be given by
\[
f(x):=
\begin{cases}
	 \sgn(p(x)) & \text{if $x\in E\setminus g^{-1}(N)$,} \\
	 0 & \text{otherwise.}
\end{cases}
\]
For this choice of $g$ and $f$, the sum between square brackets 
in formula \eqref{e-orientedarea} is a positive \emph{odd} integer 
for every $s\in g(E)\setminus N$ and is $0$ otherwise, 
and therefore \eqref{e-orientedarea} yields
\[
\int_\gamma f \, dg
\ge |g(E)\setminus N|
= |g(E)| 
>0
\, .
\]
This contradicts the assumption that $\gamma$ has degree zero 
(cf.~Remark~\ref{s-degzerorem}(i)).
\end{proof}

The following construction is used in the proof of 
Theorem~\ref{s-canopara}.

\begin{parag}[Joining paths]
\label{s-joinpath}
Let $I:=[0,1]$, let $\gamma:I\to X$ be a \emph{closed} path, 
and let $\gamma':I\to X$ be a path whose endpoint 
$\gamma'(0)$ belongs to the image of $\gamma$.
We join these paths to form a \emph{closed} path 
$\gamma\ltimes\gamma':I\to X$ as follows (see figure~\ref{figure2}): 
we choose $t_0$ such that $\gamma'(0)=\gamma(t_0)$ and set%
\,\footnoteb
{The notation $\gamma\ltimes\gamma'$ is not quite appropriate, 
because this path does not depend
only on $\gamma$ and $\gamma'$, but also on the choice of $t_0$.}
\[
\big( \gamma\ltimes\gamma'\big) (t):= 
  \begin{cases}
    \gamma(3t)        & \text{if $0\le t \le t_0/3$,} \\
    \gamma'(3t-t_0)   & \text{if $t_0/3 < t \le (t_0+1)/3$,} \\
    \gamma'(t_0+2-3t) & \text{if $(t_0+1)/3 < t \le (t_0+2)/3$,} \\
    \gamma(3t-2)      & \text{if $(t_0+2)/3 < t \le 1$.} \\
  \end{cases}
\]
\end{parag}

\begin{figure}[h]
\begin{center}
  \includegraphics[scale=1.1]{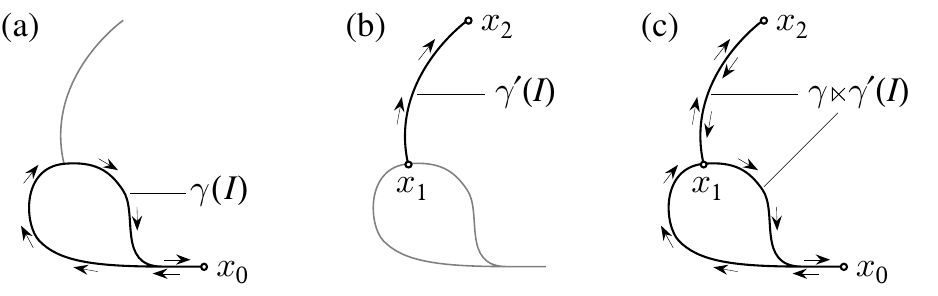}  
  \caption{Example of joined paths: $\gamma$, $\gamma'$ and 
  $\gamma\ltimes\gamma'$ are described in (a), (b), (c), respectively 
  (the arrows show how each path goes through its own image); 
  $x_0$ agrees with both endpoints of $\gamma$ and $\gamma\ltimes\gamma'$,
  $x_1$ and $x_2$ are the endpoints of $\gamma'$, and $x_1=\gamma'(0)=\gamma(t_0)$.} 
  \label{figure2}
\end{center}
\end{figure}

The next lemma collects some straightforward properties 
of $\gamma\ltimes\gamma'$ that will be used later.
We omit the proof.

\begin{lemma}
\label{s-joinpathlemma}
Take $\gamma$, $\gamma'$ and $\gamma\ltimes\gamma'$ as in Subsection~\ref{s-joinpath}.
The following statements hold:
\begin{itemizeb}
\item[(i)]
if $\gamma$ and $\gamma'$ are Lipschitz, then
$\gamma\ltimes\gamma'$ is Lipschitz;
\item[(ii)]
$\Len(\gamma\ltimes\gamma')=\Len(\gamma)+2\,\Len(\gamma')$;
\item[(iii)]
if $\gamma$ and $\gamma'$ have bounded multiplicities,
so does $\gamma\ltimes\gamma'$;
\item[(iv)]
if the path $\gamma$ has multiplicity $2$ at all points in its image
except finitely many, $\gamma'$ has multiplicity $1$ at all 
points in its image except finitely many, and the set 
$\gamma(I) \cap \gamma'(I)$ is finite, then $\gamma\ltimes\gamma'$ 
has multiplicity $2$ at all points in its image 
except finitely many;
\item[(v)]
for every $f:X\to\R$ bounded and Borel, and every 
$g:X\to\R$ Lipschitz there holds
\[
\int_{\gamma\ltimes\gamma'} f \, dg 
= \int_\gamma f\, dg
\, ;
\] 
\item[(vi)]
if $\gamma$ has degree zero 
(cf.~Subsection~\ref{s-degzero})
then $\gamma\ltimes\gamma'$
has degree zero.
\end{itemizeb}
\end{lemma}

\begin{proof}[Proof of Theorem~\ref{s-canopara}]
Let $I:=[0,1]$.
We obtain the path $\gamma:I\to X$ with the required properties 
as limit of the closed paths $\gamma_n:I\to X$ constructed 
by the inductive procedure described in the next two steps. 

\medskip
\emph{Step~1: construction of $\gamma_1$.}
We choose $x'_0\in X$ and take $x_0\in X$
which maximizes the distance from $x'_0$. 
By Proposition~\ref{s-pathconn}, there exists an
\emph{injective} Lipschitz path $\gamma'_0:I\to X$ 
that connects $x'_0$ to $x_0$.
We then set
\[
\gamma_1(t):= 
\begin{cases}
\gamma'_0(2t) & \text{if $0\le t \le 1/2$,} \\
\gamma'_0(2-2t) & \text{if $1/2 < t \le 1$.} \\
\end{cases}
\]
Note that that $\gamma_1$ is closed, has degree $0$,
and its multiplicity is $2$ at all 
points of $\gamma_1(I)$ except $\gamma_1(1/2)$, 
where it is $1$.
Clearly $\Len(\gamma_1)=2\,\Len(\gamma'_0)$.

\medskip
\emph{Step~2: construction of $\smash{\gamma_{n+1}}$, given $\gamma_n$.}
We assume that $\gamma_n(I)$ is a proper subset of~$X$.%
\footnoteb{This inductive procedure stops if $\gamma_n$ is surjective;
when this happens, we simply reparametrize $\gamma_n$ 
so that it has constant speed, and set $\gamma:=\gamma_n$.
In this case it is quite easy to verify that $\gamma$ has the required 
properties (we omit the details).}
Then we take a point $x_n\in X$ which maximizes the distance 
from $\gamma_n(I)$, and an \emph{injective} Lipschitz 
path $\gamma'_n:I\to X$ that connects $x_n$ to some point 
$x'_n\in\gamma_n(I)$. 

By ``cutting off a piece of
$\gamma'_n$'' we can assume that this path intersects 
$\gamma_n(I)$ \emph{only} at the endpoint~$x'_n$.
We can also assume that $x'_n = \gamma'_n(0)$. Then we set
\[
\gamma_{n+1} := \gamma_n \ltimes\gamma'_n
\, .
\] 

\medskip
\emph{Step~3: properties of $\gamma_n$.}
Using Lemma~\ref{s-joinpathlemma}, one easily proves 
that each $\gamma_n$ is closed and Lipschitz, has degree zero,
and satisfies
\begin{equation}
\label{e-canoparapr1}
\ell_n:=\Len(\gamma_n) = 2 \sum_{m=0}^{n-1} \Len(\gamma'_m)
\, .
\end{equation}
Moreover the multiplicity of $\gamma_n$ is bounded and equal to $2$ 
for all points in $\gamma_n(I)$ except finitely many.
This last property, together with formula \eqref{e-lengthformula},
yields
\begin{equation}
\label{e-canoparapr2}
\ell_n:=\Len(\gamma_n) \le 2\,\Haus^1(X)
\, .
\end{equation}

\medskip
\emph{Step~4: reparametrization of $\gamma_n$.}
Since the multiplicity of $\gamma_n$ is bounded, 
$\gamma_n$ is not constant on any subinterval of $I$, 
and therefore it admits a reparametrization 
with constant speed equal to $\ell_n$ (Remark~\ref{s-pathrem}(ii)).
In the rest of the proof we replace $\gamma_n$ 
by this reparametrization, which still satisfies 
all the properties stated in Step~3.

\medskip
\emph{Step~5: construction of $\gamma$.}
The paths $\gamma_n:I\to X$ are closed and uniformly 
Lipschitz, and more precisely $\Lip(\gamma_n)=\ell_n\le 2\,\Haus^1(X)$. 
Therefore, possibly passing to a subsequence, 
the paths $\gamma_n$ converge uniformly to a path $\gamma:I\to X$ 
which is closed and Lipschitz, and satisfies
$\Lip(\gamma) \le 2\,\Haus^1(X)$.

\medskip
\emph{Step~6: $\gamma$ is surjective.}
Equations \eqref{e-canoparapr1} and \eqref{e-canoparapr2} imply that
the sum of the lengths of all paths $\gamma'_n$ is finite, and then
\[
\Len(\gamma'_n) \to 0 
\quad\text{as $n\to +\infty$.}
\]
Now, recalling the choice of $x_n$ and the fact that 
$\gamma'_n$ connects $x_n$ to $\gamma_n(I)$ (cf.~Step~2) 
we obtain that
\[
d_n:=\sup_{x\in X} \dist(x,\gamma_n(I))
=\dist(x_n,\gamma_n(I)) 
\le \Len(\gamma'_n)
\, ,
\]
and therefore $d_n$ tends to $0$ as $n\to +\infty$,
which means that the union of all $\gamma_n(I)$ is dense in $X$.

Now, $\gamma_m(I)$ contains $\gamma_n(I)$ for every $m\ge n$, 
and then $\gamma(I)$ contains $\gamma_n(I)$ for every $n$.
Hence $\gamma(I)$ contains a dense subset of $X$, 
and since it is closed, it must agree with $X$.

\medskip
\emph{Step~7: completion of the proof.}
Since the paths $\gamma_n$ have degree zero, so does $\gamma$ 
(Proposition~\ref{s-degzeroprop}(ii)), and the proof of 
statement~(i) is complete. 
This fact, the surjectivity of $\gamma$, and
Proposition~\ref{s-degzeroprop}(iii) imply that 
\begin{equation}
\label{e-canoparapr3}
m(\gamma,x)\ge 2 
\quad\text{for $\Haus^1$-a.e.~$x\in X$.}
\end{equation}
On the other hand, estimate \eqref{e-canoparapr2} and the 
semicontinuity of the length (Remark~\ref{s-pathrem}(i)) imply
\begin{equation}
\label{e-canoparapr4}
\Len(\gamma)\le 2\,\Haus^1(X)
\, .
\end{equation}
Now, equations \eqref{e-canoparapr3} and \eqref{e-canoparapr4},
together with \eqref{e-lengthformula}, imply that equality must 
hold both in \eqref{e-canoparapr3} and in \eqref{e-canoparapr4}, 
and statement~(ii) is proved. 

To prove statement~(iii), note that 
$\Lip(\gamma) \le 2\,\Haus^1(X) = \Len(\gamma)$ (cf.~Step~5), 
and then $\gamma$ must have constant speed (cf.~Remark~\ref{s-pathrem}(v)).
\end{proof}

We conclude this section by another proof of \Golab's theorem.

\begin{proof}
[Second proof of Theorem~\ref{s-golab} for $\boldsymbol{m=1}$]
We must show that for every sequence of continua $K_n$ 
contained in $X$ which converge in the Hausdorff distance
to some continuum $K$, there holds 
$\liminf \Haus^1(K_n) \ge \Haus^1(K)$.

We can clearly assume that the lengths $\Haus^1(K_n)$
are uniformly bounded.
For every $n$, we apply Theorem~\ref{s-canopara}
to the continuum $K_n$ and find a path $\gamma_n:I\to X$
with $I:=[0,1]$ such that 
$\gamma_n(I)=K_n$, 
$\gamma_n$ has constant speed and degree zero, 
and $\Len(\gamma_n)=\Lip(\gamma_n)=2\,\Haus^1(K_n)$.

Note that the paths $\gamma_n$ are uniformly Lipschitz, 
and therefore, possibly passing to a subsequence, 
they converge uniformly to some path $\gamma: I\to X$, 
and clearly $\gamma(I)=K$. 

Moreover Proposition~\ref{s-degzeroprop}(ii) implies 
that $\gamma$ has degree zero, and Proposition~\ref{s-degzeroprop}(iii) 
implies that $m(\gamma,x)\ge 2$ for $\Haus^1$-a.e.~$x\in K$. 
Then formula~\eqref{e-lengthformula} implies
that $\Len(\gamma) \ge 2\,\Haus^1(K)$.

We can now conclude, using the semicontinuity of length
(cf.~Remark~\ref{s-pathrem}(i)):
\[
\liminf_{n\to+\infty} \Haus^1(K_n)
=\frac{1}{2} \liminf_{n\to+\infty} \Len(\gamma_n)
\ge \frac{1}{2} \Len(\gamma)
\ge \Haus^1(K)
\, .
\qedhere
\]
\end{proof}

\section*{Acknowledgements}
We thank Alexey Tuzhilin for pointing 
out a mistake in the proof of Lemma~\ref{s-lemmachiave1} 
and for several valuable comments.

The research of the first author has been partially supported by 
the University of Pisa through the 2015 PRA Grant
``Variational methods for geometric problems'',   
%by the Italian Ministry of Education, University and Research (MIUR) 
%through the 2011 PRIN Grant ``Calculus of variations'', 
and by the European Research Council (ERC) through the 
Advanced Grant ``Local structure of sets, measures and currents''.
	%
	%
	% references
	%
	%
\bibliographystyle{plain}

	%
	%
	% affiliations
	%
	%
\vskip .5 cm
{\parindent = 0 pt\footnotesize
G.A.
\\
Dipartimento di Matematica, 
Universit\`a di Pisa
\\
largo Pontecorvo 5, 
56127 Pisa, 
Italy 
\\
e-mail: \texttt{giovanni.alberti@unipi.it}

\bigskip
M.O.
\\
Scuola Normale Superiore
\\
piazza dei Cavalieri 7,
56126 Pisa, 
Italy 
\\
e-mail: \texttt{martino.ottolini@sns.it}
\par

\par
}

\end{document}